\documentclass[12pt]{amsart}
\DeclareFontFamily{OML}{script}{}
\DeclareFontShape{OML}{script}{m}{it}
{ <5-20> rsfs10 }{}
\DeclareMathAlphabet{\mathscript}{OML}{script}{m}{it}

\renewcommand{\mathcal}[1]{{\mathscript #1}\hspace{0.2ex}}
\usepackage{color}
\ifx\red\undefined
\newcommand{\red}{\color{red}}

\fi

\usepackage{graphicx}
\usepackage{enumitem}
\usepackage{aliascnt}
\usepackage{hyperref}

\usepackage[margin=1in]{geometry}
\usepackage{pifont}
\usepackage{amsthm}
\usepackage{tikz}

\usepackage{amscd}
\usepackage{amsmath}
\usepackage{latexsym}
\usepackage{amsfonts}
\usepackage{amssymb}
\usepackage{color}

\ifx\text\undefined
\newcommand{\text}{\mbox}
\fi
\ifx\operatorname\undefined
\newcommand{\operatorname}{\mathop}
\fi
\newcommand\be{\begin{equation}}
\newcommand\ee{\end{equation}}
\newcommand\bea{\begin{eqnarray}}
\newcommand\eea{\end{eqnarray}}
\newcommand\beaa{\begin{eqnarray*}}
\newcommand\eeaa{\end{eqnarray*}}

\newenvironment{eqa}{\begin{equation}%
  \begin{array}{rcl}}{\end{array}\end{equation}}
\newcommand\beqa{\begin{eqa}}
\newcommand\eeqa{\end{eqa}}

\numberwithin{equation}{section}

\newtheorem{theorem}{Theorem}[section]

\newtheorem{lemma}{Lemma}[section]

\newcommand{\void}[1]{}

\def\R{\mathbb{R}}

\begin{document}

\title[Some monotone properties for solutions to
 a reaction-diffusion model]{Some monotone properties for solutions to
 a reaction-diffusion model}

\author{Rui Li}
\date{\today}
\address{Institute for Mathematical Sciences, Renmin  University of China,
Beijing, 100876, P.R. China}
\email{liruicxis@ruc.edu.cn}
\author{Yuan Lou}
\address{Department of Mathematics, Ohio State University,
Columbus, OH 43210, USA}
\email{lou@math.ohio-state.edu}
\thanks{Corresponding author: Rui Li, liruicxis@ruc.edu.cn}

\begin{abstract}
Motivated by the recent investigation of a predator-prey model in heterogeneous
environments \cite{LouYuan-WangBiao}, we show that the maximum of the unique
positive solution of the scalar equation
 \begin{equation}\label{eq:01}\begin{cases}
\mu\Delta\theta+(m(x)-\theta)\theta=0 \hspace{0.5em} &\text{in}\hspace{0.5em}\Omega,\\
\frac{\partial \theta}{\partial n}=0 \hspace{0.5em} &\text{on}\hspace{0.5em}\partial\Omega
\end{cases}\end{equation}
is a strictly  monotone decreasing function of the diffusion rate $\mu$
for several classes of function $m$, which substantially
improves a result in \cite{LouYuan-WangBiao}.
However,
the minimum of the positive solution of \eqref{eq:01}
is not always  monotone increasing in the diffusion rate \cite{HeXiaoqing-NiWeiMing2016}.
\end{abstract}
\maketitle

\section{\bf Introduction}

Consider the scalar reaction-diffusion equation
\begin{equation}\label{genf}\begin{cases}
u_t=\mu \Delta u+f(x, u)\quad &\mbox{in}\ \Omega\times (0, \infty),\\
\frac{\partial u}{\partial n}=0 \quad &\mbox{on}\ \partial \Omega\times (0, \infty),\\
u(x,0)=u_0(x) \quad &\mbox{in}\ \Omega,
\end{cases}\end{equation}
where $u(x,t)$ represents the density of the species at location $x$
and time $t$,
$\mu>0$ is  the diffusion rate, $\Omega$ is the habitat of  species
 and it is assumed to be an open bounded region in $R^{N}$ with smooth boundary $\partial{\Omega}$,
 and $n$ is the outward unit normal vector on $\partial{\Omega}$.
 The zero-flux boundary condition means there are no individuals crossing the boundary of the habitat.

In the last few decades   equation \eqref{genf}
 has attracted considerable attentions as an important single species model in spatial ecology; see
 \cite{CC1989, CC1991, CC1993, CC1998,
 Cantrell-Cosner, HeNi1, HeNi2, HeNi3, lou2007, lou2015, Ni2011} and the references therein.
In the case of $f(x,u)=f(u)$, i.e. the underlying environment
is spatially homogeneous,
 it was shown in \cite{CaHo, Matano}
  that any stable steady state of \eqref{genf} must be constant,
   i.e. they are independent of $x$ and $\mu$. However, if $f(x,u)$
    depends on $x$,
    then \eqref{genf}
   could have non-constant stable steady state solutions, which
   are also dependent on $\mu$.
   It is easy to see that any bounded steady state of \eqref{genf}
    converges to some constant as $\mu\to\infty$, which
   is not surprising as diffusion generally tends to average the distribution
of organisms, i.e. increasing the diffusion will reduce the spatial variability
of population distributions.

A natural question  aries: What kind of monotone property
     holds for steady state of \eqref{genf} in terms of parameter $\mu$?
     It was shown in \cite{LandL} that
     for any stable  steady state of \eqref{genf},
   denoted by $u(x; \mu)$, it holds that $\int_\Omega |\nabla u|^2$
   is monotone decreasing in $\mu$.
     Biologically this implies that the population distribution
becomes flatter in average if we increase the diffusion rate.
     Recently, in the investigation of a predator-prey model in heterogeneous
environments \cite{LouYuan-WangBiao},
the authors studied another monotone property of steady state of \eqref{genf};
Namely, whether the maximum of the unique
 solution of the following equation, denoted by $\theta(x; \mu)$,
 is  monotone decreasing in $\mu$:
\begin{equation}\label{EQmain}\begin{cases}
\mu \Delta \theta+ \theta(m(x)-\theta)=0\quad \hspace{0.5em} &\mbox{in}\ \Omega,\\
\theta>0 \quad \hspace{0.5em} &\mbox{in}\ \Omega,\\
\frac{\partial\theta}{\partial n}=0\quad \hspace{0.5em} &\mbox{on}\ \partial \Omega.
\end{cases}\end{equation}

 In order to ensure the existence and uniqueness of solution
 of \eqref{EQmain} for all $\mu>0$,  throughout the paper we always assume that

\medskip
\begin{enumerate}[label=(M\arabic*), start=0] 	
\item \label{conditions}$m(x) \in C^{1}(\bar{\Omega})$, it is
non-constant and $\int_{\Omega} m(x)dx\ge 0$.
	\end{enumerate}
\medskip

The proof of existence and uniqueness for  solution
 of \eqref{EQmain} can be found in \cite{Cantrell-Cosner}.
 Under the assumption (M0), the following result was established in \cite{LouYuan-WangBiao}:
\begin{lemma}\label{lem:LW}
	Suppose that $\Omega$ is an interval, $m(x)\in C^{2}(\bar{\Omega}), m_x (x)\neq 0$ and $m_{xx}(x)\neq 0$ in $\bar{\Omega}$. Then $\max_{x\in\bar{\Omega}} \theta(x; \mu) $ is strictly decreasing in $\mu$.
	\end{lemma}

Lemma \ref{lem:LW} plays an important role in understanding the dynamics of
the predator-prey model considered in \cite{LouYuan-WangBiao};
See \cite{JHe} for further developments.
The issue of the monotonicity of $\max_{x\in\bar{\Omega}} \theta(x; \mu) $
also appeared in a recent work \cite{HLLN} on consumer-resource dynamics
in heterogeneous environments.

Our main goal in this paper is
to extend Lemma \ref{lem:LW}.
Our first result concerns general domain
and assumes the following condition:

\medskip
\begin{enumerate}[label=(M\arabic*),resume]
	\item \label{conditionc} $m$ is positive in $\bar\Omega$ and satisfies
$
{\max_{\bar{\Omega}} m}\leq 2 {\min_{\bar{\Omega}} m}.
$
\end{enumerate}
\medskip

\begin{theorem}\label{doublem}
	Assume that $m(x)$ satisfies {\rm\ref{conditions}} and {\rm\ref{conditionc}}. Then the function
	$$\mu\mapsto M(\mu):=\max_{x\in\bar{\Omega}}\theta(x;\mu)$$
	is strictly decreasing in $\mu\in(0,\infty)$ and the function
	$$\mu\mapsto S(\mu):=\min_{x\in \bar{\Omega}}\theta(x;\mu)$$
	is strictly increasing in $\mu\in(0,\infty)$.
\end{theorem}

	Theorem {\rm\ref{doublem}} not only infers
 that  $M(\mu)$ is decreasing in $\mu$, but
  also $S(\mu)$ is increasing in $\mu$.
 For general $\Omega$ and $m$, it is unknown whether $M(\mu)$ is always decreasing.
 However, $S(\mu)$
is not always  monotone increasing in the diffusion rate \cite{HeXiaoqing-NiWeiMing2016}.
We suspect that  $M(\mu)-S(\mu)$, which
measures the spatial variation of the population distribution,
 is always decreasing.

Next we consider one-dimensional domain and
monotone  $m(x)$.

\medskip
\begin{enumerate}[label=(M\arabic*), resume]
	\item \label{conditiona} $m^{+}(x):=\max\{m(x), 0\}$ is monotone in $[0,1].$
\end{enumerate}
\medskip

\begin{theorem}\label{thmmono}	
			Suppose $m$ satisfies {\rm\ref{conditions}} and {\rm\ref{conditiona}},	then
		$M(\mu)$ is strictly decreasing.
More precisely, if $m^{+}(x)$ is non-decreasing in $(0,1)$, then $\theta'(x;\mu)>0$ for $x\in(0,1)$ and $\theta_{\mu}(1;\mu)<0;$ If $m^{+}(x)$ is non-increasing in $(0,1)$,
then $\theta'(x;\mu)<0$ for $x\in (0,1)$ and $\theta_{\mu}(0;\mu)>0.$
	\end{theorem}

Theorem \ref{thmmono} improves
Lemma \ref{lem:LW}	by dropping the condition $m_{xx}(x)\neq0$ in $\Omega$.

Our final result says that
if $m$ has a unique interior critical point in one dimensional $\Omega$,
then $M(\mu)$ is also monotone decreasing.

\medskip
\begin{enumerate}[label=(M\arabic*), resume]
	\item \label{conditionb}
For some $\rho\in (0,1)$, $m'(x)>0$ in $[0,\rho)$ and $m'(x)<0$ in $(\rho, 1]$.
\end{enumerate}

\begin{theorem}\label{deincreasing}
 Suppose that $m(x)$ satisfies  {\rm\ref{conditions}} and
 {\rm\ref{conditionb}},  then  $M(\mu)$ is strictly decreasing in $\mu$.
\end{theorem}

This paper is organized as follows. In Section 2 we establish Theorem \ref{doublem}.
Section 3 is devoted to the proofs
of Theorems \ref{thmmono} and \ref{deincreasing}.
Finally in Section 4 we discuss our main results and possible extensions.

\section{\bf General domain}
The goal of this section is to establish Theorem \ref{doublem}.
We will also illustrate that
$\min_{\bar\Omega}\theta$ is not necessarily monotone increasing in $\mu$, a result
due to He and Ni \cite{HeXiaoqing-NiWeiMing2016}.
To this end we first establish some properties of the solution of \eqref{EQmain}.

\begin{lemma}\label{maximum}
	Suppose  $m(x)$ satisfies {\rm\ref{conditions}}, then \eqref{EQmain} admits a unique solution, denoted by $\theta=\theta(x; \mu)$.
Furthermore,
\begin{equation}\label{eq:2-1-01}
\min_{\bar{\Omega}} m^{+}<\theta(x;\mu)<\max_{\bar{\Omega}} m
\end{equation}
holds	for all $x\in\bar{\Omega}$ and $\mu>0$.
\end{lemma}
 The existence and uniqueness of the solution
 of  \eqref{EQmain}  is standard, see \cite{Cantrell-Cosner}.
 The proof of \eqref{eq:2-1-01} is also known and it
 follows from the maximum principle; See \cite{Averill-Lam-LouYuan}
 and references therein.

\begin{lemma}\label{lemmafor}
	Assume that $m$ satisfies  {\rm\ref{conditions}} and {\rm\ref{conditionc}}. Then
	\begin{equation}\label{doublerelationship}
	\min_{x\in \Omega} \theta<\theta +\mu \frac{\partial\theta}{\partial\mu}<\max_{x\in\Omega} \theta
\hspace{0.5ex}\quad \mbox{in}\hspace{0.5ex}\ \bar\Omega.
	\end{equation}
	In particular, we have:

\smallskip
	
\noindent{\rm(i)} If $\bar{x}$ is a global maximum point of $\theta$, then
	$$
\frac{\partial\theta}{\partial\mu}(\bar{x}; \mu)<0;
$$
	
\noindent{\rm(ii)} If $\bar{x}$ is a global minimum point of $\theta$, then
	$$
\frac{\partial\theta}{\partial\mu}(\bar{x}; \mu)>0.
$$

\end{lemma}

\begin{proof}

 Denote $\partial \theta/{\partial \mu}$ by $\theta_{\mu}$.
Differentiating  \eqref{EQmain} with respect to $\mu,$
we derive
	\begin{equation}\label{equation31-0}\begin{cases}
	\mu \Delta\theta_{\mu}+(m-2\theta)\theta_{\mu}=-\Delta\theta
\quad &\mbox{in}\ \Omega,\\
	\frac{\partial\theta_{\mu}}{\partial n}=0
\quad &\mbox{on}\ \partial\Omega.
	\end{cases}\end{equation}

 For any constant $K$, set
$$
v:=\theta+\mu\theta_{\mu}-K.
$$
By direct computation and applying \eqref{EQmain} and \eqref{equation31-0},
$v$ satisfies
$$\begin{cases}
\mu \Delta v+(m-2\theta)v=(m-2\theta)(\theta-K) \quad \mbox{in}\ \Omega,
\\
\frac{\partial v}{\partial n}=0 \quad \mbox{on}\ \partial\Omega.
\end{cases}
$$

By assumption {\rm\ref{conditionc}} and Lemma \ref{maximum}, we have
\begin{equation}\label{eq:2-1-1}
\max m\le 2\min m<2\theta \quad \mbox{in}\ \bar\Omega;
\end{equation}
i.e.  $m-2\theta<0$ holds in $\bar\Omega$.

We first establish $\theta +\mu \theta_{\mu}<\max_{\bar\Omega} \theta$
in $\bar\Omega$.
 For this case, choosing $K=\max_{\bar\Omega} \theta$
we see that
$v$ satisfies
\begin{equation}\label{eq:2-1-2}
\mu \Delta v+(m-2\theta)v=(m-2\theta)(\theta-\max_{\bar\Omega} \theta)\ge 0
 \quad \mbox{in}\ \Omega,
\end{equation}
where the last inequality follows from Lemma \ref{maximum}
and \eqref{eq:2-1-1}.

It suffices to show $v<0$ in $\bar\Omega$:
We argue by contradiction and suppose that
the maximum of $v$ is non-negative, and it is attained at some
$\bar{x}\in\bar\Omega$.
If $\bar{x}\in\Omega$, by \eqref{eq:2-1-1},
\eqref{eq:2-1-2} and the strong maximum principle \cite{PW},
$v$ is constant in $\Omega$.
Again by  \eqref{eq:2-1-1} and
\eqref{eq:2-1-2}, we have $\theta-\max_{\bar\Omega} \theta=v$
is constant, which is a contradiction as $\theta$ is non-constant.
  Hence we may assume
that $\bar{x}\in\partial\Omega$.
 As $v$ is non-constant in $\Omega$ and $v(\bar{x})\ge 0$,
by the Hopf's boundary point lemma \cite{PW},
$\frac{\partial v}{\partial n}(\bar{x})>0$,
which contradicts the boundary condition of $v$.
Therefore,  there is no non-negative maximum of $v$,
i.e. $\max_{\bar\Omega} v<0$. This proves $\theta +\mu \theta_{\mu}<\max_{\bar\Omega} \theta$
in $\bar\Omega$, from which part (i) follows immediately.

The second part can be proved similarly by choosing $K=\min_{\bar\Omega} \theta$.
	\end{proof}

\medskip

\noindent{\textbf{Proof of Theorem \ref{doublem}.}}
  \quad For  any fixed $\bar{\mu}>0$, we prove that there exists some $\delta>0$ such that
	$$M(\bar{\mu})<M(\mu), \quad \mu\in(\bar{\mu}-\delta, \bar{\mu}).
$$
Let $\bar{x}$ be a global maximum point of $\theta(\cdot; \bar{\mu})$. Then by Lemma \ref{lemmafor}, we have
$$\theta_{\mu}<0 \hspace{0.5ex}\quad \text{at} \hspace{0.5ex} x=\bar{x},
\quad \mu=\bar{\mu}.$$
By the continuity of $\theta_{\mu}$,  there exists some small $\delta>0$ such that
$$\theta_{\mu}<0\quad  \hspace{0.5ex}\text{for} \hspace{0.5ex} |x-\bar{x}|\leq \delta,
\ |\mu-\bar{\mu}|\leq \delta.$$
Thus 
$$\theta(x;\bar{\mu})<\theta(x,\mu) \quad \hspace{0.5ex}\text{for} \hspace{0.5ex} |x-\bar{x}|\leq \delta, \ -\delta<\mu-\bar{\mu}<0.$$
In particular,
$$M(\bar{\mu})=\theta(\bar{x};\bar{\mu})<\theta(\bar{x};\mu)\le M(\mu)\quad  \hspace{0.5ex}\text{for} \hspace{0.5ex} -\delta<\mu-\bar{\mu}<0.$$
This proves the assertion. Hence $\mu\longmapsto M(\mu)$ is strictly decreasing for $\mu\in (0,\infty)$.
By using the same method, we can show that $\mu\longmapsto S(\mu)$ is strictly increasing for $\mu\in (0,\infty)$.
\hfill$\Box$

\medskip

For the rest of this section
we illustrate that $\min_{\bar\Omega}\theta$ is not necessarily monotone increasing in $\mu$.
This is due to  \cite{HeXiaoqing-NiWeiMing2016}.
To this end we focus on the case of sufficiently large $\mu$.
For convenience, we consider
\begin{equation}\label{LiRui-41-1}
\begin{cases}
\Delta u + \lambda u(m(x)-u)=0 \hspace{0.5em} &\mbox{in} \ \Omega, \\
u>0 \hspace{0.5em} &\mbox{in} \ \Omega, \\
\frac{\partial u}{\partial n}=0\hspace{0.5em} &\mbox{on}\ \partial\Omega,
\end{cases}
\end{equation}
where  $\lambda=1/\mu$ and $u(x; \lambda)=\theta(x; \mu)$.

We first state
the following result
of He and Ni (Proposition 3.1, \cite{HeXiaoqing-NiWeiMing2016}).
For the convenience of readers we include a proof.

\begin{lemma}\label{lem400}
	Suppose $m$ satisfies {\rm\ref{conditions}}. Then there exists a family of positive solutions $u=u(x; \lambda)$
of \eqref{LiRui-41-1} which
 are smooth in $\lambda$ for $|\lambda|\ll1$. Moreover,
	$$u(x; \lambda)=\bar{m}+\lambda (C(m)+\rho_{m}(x))+O(|\lambda|^2)$$
	as $|\lambda|\rightarrow 0$,
 where $\rho_{m}$ and $C(m)$ are uniquely determined by
	\begin{equation} \label{LiRui-43rho}\begin{cases}
	\Delta \rho_{m} + \bar{m}(m(x)-\bar{m})=0 \text{ in }\Omega,\\
	\frac{\partial \rho_{m}}{\partial n}=0  \text{ on }\partial\Omega, \\
    \int_{\Omega}\rho_{m} dx=0,\\
	C(m)=-\frac{1}{\bar{m}|\Omega|}\int_{\Omega}(m-\bar{m})\rho_{m} dx=\frac{\int_{\Omega}|\nabla\rho_{m}|^2 dx}{\bar{m}^{2}|\Omega|}.
	\end{cases}\end{equation}
\end{lemma}

\begin{proof}
	For $\alpha\in (0, 1)$, set
	$$
\aligned
C^{2,\alpha}_{N}(\overline{\Omega})&=\{u\in C^{2,\alpha}(\overline{\Omega}):
\frac{\partial u}{\partial n}=0\text{ on }\partial\Omega\};
\\
X&=\{u\in C^{2,\alpha}_{N}(\overline{\Omega}): \int_{\Omega} udx=0\}; 
\\
Y&=\{u\in C^{\alpha}(\overline{\Omega}): \int_{\Omega} udx=0\}.
\endaligned
$$
	
Set $u=w+s$, $w\in X$, $s\in\R$. Then equation \eqref{LiRui-41-1} is reduced to solving
	$$\Delta w + \lambda (w+s)(m(x)-w-s)=0\text{ in }\Omega,
\quad \frac{\partial w}{\partial n} =0\text{ on }\partial\Omega$$
	with $$\int_{\Omega}(w(x)+s)(m(x)-w(x)-s)dx=0.$$
	
Define $F:X\times\R\times\R\rightarrow Y\times\R$ by
	$$F(w,s,\lambda)=\left(\begin{array}{c} \Delta w + \lambda (w+s)(m(x)-w-s) \\   \int_{\Omega}(w(x)+s)(m(x)-w(x)-s)dx \end{array}\right).$$
Clearly  $(u, \lambda)$ is a solution of \eqref{LiRui-41-1} with $\lambda\neq0$ if and only if
	$$w=u-s, \ s=\frac{1}{|\Omega|}\int_{\Omega}udx, \ \lambda\neq0$$
	is a solution to $F=0$.
	
	We next verify the nonsingular condition of $F$ with respect to the variables $(w,s)$ at $(w,s,\lambda)=(0,\bar{m},0)$.
	Note that $F(0,\bar{m},0)=0$ and the Fr\'echet
derivative $D_{(w,s)}F$ at $(w,s,\lambda)=(0,\bar{m},0)$,  denoted by $L:X\times\R\rightarrow Y\times\R$, is given by
	$$L(\phi,\tau)=\left(\begin{array}{c} \Delta \phi \\   \int_{\Omega}(m(x)-2\bar{m})(\phi+\tau)dx  \end{array}\right).$$
	
To show that $L$ is nonsingular, we consider the equation $L(\phi,\tau)=(f,h)\in Y\times\R$:
	\begin{equation} \label{LiRui-44}\begin{cases}\Delta \phi=f\text{ in }\Omega,\\
  \frac{\partial\phi}{\partial n}=0\text{ on }\partial\Omega,\\
    \int_{\Omega}\phi dx=0,\\
	\int_{\Omega}(m(x)-2\bar{m})(\phi+\tau)dx=h.
\end{cases}\end{equation}
	As $f\in Y$, we see that $\phi\in X$ exists and is unique, and
	$$\tau=\frac{1}{\bar{m}|\Omega|}\left[\int_{\Omega}m\phi dx-h\right].$$
	Thus $L$ is surjective and is nonsingular by the Fredholm alternative. Therefore, the
implicit function theorem can be applied to show that near $(w,s,\lambda)=(0,\bar{m},0)$,
the solution to $F=0$ is uniquely given by
	$$w=w(\lambda),\ s=s(\lambda), \quad |\lambda|<\delta$$
	 for some $\delta>0$.
	 Now we turn to determine the limiting behavior near $\lambda=0$ (i.e. large diffusion).
	Note that $u|_{\lambda=0}=\bar{m}$, and $u_{\lambda}$, the derivative of $u$
 with respect to $\lambda$, satisfies
\begin{equation}
\begin{cases}
\Delta u_{\lambda} + \lambda (m-2u)u_{\lambda} + u(m(x)-u)=0 \quad \mbox{in}\ \Omega,\\
\frac{\partial u_\lambda}{\partial n}=0 \quad \mbox{on}\ \partial\Omega,\\
\int_{\Omega}(m(x)-2u)u_{\lambda}dx=0.
\end{cases}
\end{equation}
	Set $\lambda=0$, we know $u_{\lambda}|_{\lambda=0}=C(m)+\rho_{m}$,
 where $\rho_{m}$ and $C(m)$ are defined in \eqref{LiRui-43rho}. 
\end{proof}
By the proof of Lemma \ref{lem400} and Proposition 1.5 in \cite{HeXiaoqing-NiWeiMing2016},
there exists a $C^{1}$ function $m$ such that $C(m)+\rho_{m}(x)>0$ in $\overline{\Omega}$, i.e.
$$u_{\lambda}(x;0)=C(m)+\rho_{m}(x)>0,\ x\in\overline{\Omega}.$$
In particular,
$u_{\lambda}(x;\lambda)>0$ for all $x\in\overline{\Omega}$
and $|\lambda|\ll1.$
This implies that $\min_{x\in\bar\Omega}\theta(x; \mu)$ is
strictly decreasing for sufficiently large $\mu$.

\section{\bf One dimensional domain}
	In this section, we always
assume that the domain $\Omega$ is an open bounded interval. Without loss of generality
we assume that $\Omega=(0,1)$. Our goal is to prove Theorems \ref{thmmono}
and \ref{deincreasing}.

For one-dimensional domain, \eqref{equation31-0} can be rewritten as
	\begin{equation}\label{equation31}\begin{cases}
	\mu \theta_{\mu}''+(m-2\theta)\theta_{\mu}=-\theta'', \quad 0<x<1, \\
	\theta'_{\mu}(0)=\theta'_{\mu}(1)=0.
	\end{cases}\end{equation}
To study \eqref{equation31}, we consider the problem
\begin{equation}\label{general2}\begin{cases}
v''+c(x)v=h', \quad 0\le x\le t_0,\\
v(0)=v'(0)=0,
\end{cases}\end{equation}
where $t_0$ is some positive constant.

\begin{lemma}\label{gelem}
	Assume $c(x)\in C[0,t_0], t_0>0$, $0\leq h(x)\in C^{1}[0,t_0], h(0)=0, h(x)\not\equiv0$ in any neighborhood of $x=0$. If $v$ is the solution of \eqref{general2}, then for some $\delta>0$,
	$$v(x)>0,\quad  x\in(0,\delta].$$
\end{lemma}
\begin{proof}
	Let $\phi$ and $\psi$ be the solutions of
	$$	\begin{aligned}
	\psi''+c(x)\psi=0, \quad \psi(0)=0, \psi'(0)=1;\\
	\phi''+c(x)\phi=0, \quad \phi(0)=1,\phi'(0)=0.
	\end{aligned}$$
It is easy to get
$$\psi'(y)\phi(y)-\psi(y)\phi'(y)=1.$$
Then by the method of variation of constant, we have
	\begin{equation}\begin{aligned}
	v(x)=&\int_0^{x}h'(y)[\psi(x)\phi(y)-\psi(y)\phi(x)]dy\\
=&\int_0^{x}h(y)G(x,y)dy,
\end{aligned}\end{equation}
where $G(x,y)=\phi(x)\psi'(y)-\psi(x)\phi'(y)$.  Since $G(0,0)=1,$  we get for some $\delta>0,$
$$G(x,y)>0, \hspace{0.5em} x, y\in[0,\delta],$$
then the conclusion is proved.
\end{proof}

\subsection{Monotone function $m$}
The goal of this subsection is to establish Theorem \ref{thmmono}
for monotone $m$.

 \begin{lemma}\label{lemfirst}
Suppose that  $m$ satisfies  {\rm\ref{conditionc}} and {\rm\ref{conditiona}},
then $\theta'(x)\neq 0$ in $(0,1)$ and $\theta''(x)\neq 0$ at \{0,1\}.
More precisely,  if $m^{+}(x)$ is non-decreasing, then $\theta'(x)>0$ in $(0,1)$;
if $m^{+}(x)$ is non-increasing, then $\theta'(x)<0$ in $(0,1)$.	
	\end{lemma}
\begin{proof}
	Here we only consider the case of non-decreasing
$m^{+}(x)$. We
argue by contradiction and suppose that
$\theta'(x^*)\le 0$  for some $x^*\in (0, 1)$.
	
\medskip
\noindent{}\textbf{Claim}.
There exists some $a\in (0, 1)$ such that  $\theta'(a)=0$
and $\theta'(x)>0$ for $x\in (0,a)$.
\medskip
	
By \eqref{EQmain} we have
	$$\theta''(0)=\frac{(\theta(0)-m(0))\theta(0)}{\mu}.$$
    Since $m$ is non-decreasing, by Lemma \ref{maximum}
     we have $m(0)=\min_{x\in[0,a]} m(x)<\theta(0)$, i.e.
     $\theta''(0)>0$.
     As $\theta'(0)=0$,
      we obtain $\theta'(x)>0$ for $x>0$ small.
      Recall that $\theta'(x^*)\le 0$  for some $x^*\in (0, 1)$.
      Let $a\in (0, 1)$ denote the smallest positive root
      of $\theta'(x)=0$. This proves the assertion.
    	
Next, we consider
	\begin{equation}\label{EQinterval}\begin{cases}
	\mu\theta''+\theta(m(x)-\theta)=0, \hspace{1em} x\in (0,a),\\
	\theta'(0)=\theta'(a)=0.
	\end{cases}\end{equation}
By Lemma \ref{maximum} we conclude
$$\min_{[0,a]} m^{+}< \theta(x) <\max_{[0,a]} m.$$
In particular,  $\theta(a)<\max_{x\in[0,a]} m(x)=m(a)$,
which implies that  $\theta''(a)<0$.
 By the continuity,  $\theta'(x)<\theta'(a)=0$ in $(a, a+\epsilon]$
 for some small $\epsilon>0$.
 Since $\theta'(1)=0$, there exists some
 $b\in (a+\epsilon, 1]$ such that $\theta'(b)=0$
 and $\theta'(x)<0$ in $(a, b]$.
  As $m$ is non-decreasing,  $m(x)-\theta(x)>0$ for all $x\in (a,b]$. Therefore $\theta''(x)<0$ and
  thus $\theta'(x)<0$ for all $x\in(a,b]$, which contradicts $\theta'(b)=0$.
   This proves $\theta'(x)>0$ for all $x\in (0,1)$.
	\end{proof}

\medskip

\noindent{\textbf{Proof of Theorem \ref{thmmono}}}.  By Lemma \ref{lemfirst}, when $m^{+}(x)$ is non-decreasing in $(0,1)$, we have $\theta'(x;\mu)>0$
for $x\in (0, 1)$ and $\mu>0$.

Multiplying  \eqref{equation31} and \eqref{EQmain} by $\theta$ and $\theta_{\mu}$ respectively, we get
\begin{equation}\label{equ32}\mu(\theta'_{\mu}\theta-\theta'\theta_{\mu})'=-\theta\theta''+\theta^{2}\theta_{\mu}.\end{equation}
 Integrating \eqref{equ32} over [0,1],  by $\theta'_{\mu}=\theta'=0$ at $x=0,1,$ we obtain
 \begin{equation}
 0=\int_0^{1} [-\theta\theta''+\theta^{2}\theta_{\mu}]dx,
 \end{equation}
 i.e.
 \begin{equation}\label{equationa}-\int_0^{1} \theta^{2}\theta_{\mu}dx=\int_0^{1}(\theta')^{2}dx.\end{equation}
This shows that $\theta_{\mu}$  must be negative at some point in $[0,1],$ since  $\theta$ is
a non-constant positive function in $x$.

Now we turn to prove that $\theta_{\mu}(1;\mu)<0.$ If this were false, then $\theta_{\mu}(1;\mu)\geq 0.$ Thus, in the case $\theta_{\mu}(1;\mu)>0,$ $\theta_{\mu}>0$ in some neighborhood of $x=1$ by continuity.
For the case $\theta_{\mu}(1;\mu)=0,$ evaluating \eqref{equ32} at $x=1$
and applying \eqref{EQmain} yields
$$\theta''_{\mu}(1)=-\frac{\theta''(1)}{\mu}=\frac{[m(1)-\theta(1)]\theta(1)}{\mu^{2}}>0,$$
 where the last inequality follows from
 $m(1)=\max m>\max \theta=\theta(1)$. Hence in both cases above,
  $\theta_{\mu}$ is positive in a left neighborhood of $x=1.$
  If $\theta_{\mu}$ has a zero in $[0,1),$ then there exists some $x_0<1$ such that
$$\theta_{\mu}(x)>0, x\in (x_0, 1), \quad \theta_{\mu}(x_0)=0, \quad \theta'_{\mu}(x_0)\geq 0.$$

Now integrating \eqref{equ32} over $(x_0, 1),$ we conclude
$$0\geq-\mu\theta'_{\mu}(x_0)\theta(x_0)=\theta(x_0)\theta'(x_0)+\int_{x_0}^{1}(\theta'^2+ \theta^2 \theta_{\mu})dx>0,$$
which is a contradiction. Thus the zero of $\theta_{\mu}$ in [0,1) cannot exist. So
 $\theta_{\mu}(x)>0$ for $x\in (0,1),$
which contradicts \eqref{equationa}.
Therefore, $\theta_{\mu}(1; \mu)<0$ for any $\mu>0$.
The rest of the proof is the same as that of Theorem \ref{doublem}
and is thus omitted.

 The case of non-increasing $m$ can be proved similarly. \hfill $\Box$

 \subsection{The case  of increasing-decreasing $m$}
 \begin{lemma}\label{once}
  Assume that $m$ satisfies {\rm\ref{conditions}} and {\rm \ref{conditionb}}. 
  Then $\theta'$ change signs at most once.
Furthermore, 
the interior critical of $\theta$, whenever it exists, is the unique,
non-degenerate local maximum (and thus the unique global maximum).
\end{lemma}
\begin{proof}
If $\theta$ has no critical points in $(0, 1)$, then either $\theta'(x)>0$
  for all $x\in (0, 1)$ or $\theta'(x)<0$
  for all $x\in (0, 1)$. Hence it suffices to assume that
  $\theta'$ has at least one zero in $(0, 1)$.

\medskip
\noindent{\bf  Claim}. $\theta'$ has at most one zero in $(0, \rho]$
  and at most one zero in $[\rho, 1)$.
\medskip

  To establish our assertion, we argue by contradiction: Suppose that there
  exist $x_1, x_2$ such that $0<x_2<x_1\le \rho$ and
  $\theta'(x_1)=\theta'(x_2)=0$.
  Observe that $\theta$ satisfies
  $$\mu\theta''+(m-\theta)\theta=0 \quad \ \mbox{in}\ (0, x_1),
  \quad \hspace{2ex} \theta'(0)=\theta'(x_1)=0.$$
  As $m$ is increasing in $(0, x_1)$,
  by Lemma \ref{lemfirst} we obtain $\theta'(x)>0$
  for all $x\in (0, x_1)$, which contradicts $\theta'(x_2)=0$.
  Similarly, $\theta'$ has at most one zero
   in $[\rho, 1)$.

  If $\theta'$ has a zero in $(0,\rho]$, denote the unique zero by $x_{1}$.
   By our assertion and Lemma \ref{lemfirst},
    $\theta'(x)>0$ in $(0, x_1)$ and $\theta''(x_1)<0$, thus $\theta'(x)<0$ in $(x_1, \rho]$.
    In particular, $\theta'(x)<0$ for $x>\rho$ and $x$ is close to $\rho$.
    If $\theta'$ has no zero in $[\rho, 1)$, that means
     $\theta'(x)<0$ for $x\in [\rho, 1)$, then we complete the proof of  Lemma \ref{once}.
    Hence, it remains to rule out the possibility that
    $\theta'$ also has one zero in $[\rho, 1)$,
    which we denote as $x_3$. By Lemma \ref{lemfirst},
   $\theta'<0$ in $(x_3, 1)$ and $\theta''(x_3)<0$, furthermore, $\theta'>0$ in $[\rho, x_3)$.
    However, this contradicts the fact that $\theta'(x)<0$ for $x>\rho$ when $x$
    is close to $\rho$.

    Similarly, if $\theta'$ has a zero in $[\rho,1)$, we can obtain the same conclusion.

    In summary, if $\theta'$ has at least one zero in $(0, 1)$,
    then it is unique (denoted by $x_1$):
    $\theta'>0$ in $(0, x_1)$, $\theta'<0$
    in $(x_1, 1)$ and $\theta''(x_1)<0$.
  This completes the proof.
    \end{proof}
%
%
%

\begin{lemma}\label{prepare}
	Suppose $m(x)$ satisfies {\rm\ref{conditions}}. If $\theta$ satisfies
	$$\theta'(x)\geq, \not\equiv 0, x\in(0,\eta), \quad \theta'(x)\leq,\not\equiv 0, x\in(\eta, 1)$$
	for some $\eta\in [0,1]$. Then  $\theta_{\mu}$ satisfies
	$$\theta_{\mu}(\eta)< 0.$$
	\end{lemma}	
	\begin{proof}
		
We argue by contradiction and suppose that $\theta_{\mu}(\eta)\geq0$.
			
\medskip
\noindent{\textbf{Claim}.}  $\theta_{\mu}'(\eta )\geq0$ implies that $\theta_{\mu}\equiv0$ in $[\eta ,1]$.
\medskip
			
			Otherwise, we assume that $\theta_{\mu}'(\eta )\geq0$ and $\theta_{\mu}$ does not vanish completely in $[\eta ,1]$.
			By integrating \eqref{equ32} over $[\eta ,1]$, we deduce that
			$$\int_{\eta }^{1}\theta^2\theta_{\mu}dx=-\mu\theta_{\mu}'(\eta )\theta(\eta )-\int_{\eta }^{1}(\theta')^{2}dx\leq0.$$
			Combing this with the assumption that $\theta_{\mu}\not\equiv0$ in $[\eta ,1]$, we conclude that $\theta_{\mu}$ must be negative at some point in $[\eta ,1]$.
			
			To determine the sign of $\theta_{\mu}$, we consider three cases: (i) $\theta_{\mu}(\eta )>0$; (ii) $\theta_{\mu}(\eta )=0<\theta_{\mu}'(\eta )$; (iii) $\theta_{\mu}(\eta)=0= \theta_{\mu}'(\eta )$.
			In the third case, applying Lemma \ref{gelem} to $\theta_{\mu}$-equation with the fact that $\theta_{\mu}\not\equiv0$, $\theta'\leq0$ in $[\eta ,1]$, we derive that
			\begin{equation}\label{LiLouL59}
			\theta_{\mu}(x)\geq0, x\in(\eta ,x_{1})\mbox{ and }\theta_{\mu}(x_{1})>0\mbox{ for some }x_{1}\in(\eta ,1].
			\end{equation}
			It is obvious that \eqref{LiLouL59} also holds true in case (i) or case (ii). Therefore, there
 exists some $x_{2}\in(\eta ,1)$ such that
			$$\theta_{\mu}\geq,\not\equiv0, \mbox{ in }(\eta ,x_{2})\mbox{ and }\theta_{\mu}(x_{2})=0.$$
			Hence $\theta_{\mu}'(x_{2})\leq0$. Now integrating \eqref{equ32} over $[\eta ,x_{2}]$, we get
			$$0\ge \mu[\theta_{\mu}'(x_{2})\theta(x_{2})-\theta_{\mu}'(\eta )\theta(\eta )]=-\theta'(x_{2})\theta(x_{2})+\int_{\eta }^{x_{2}}[(\theta')^{2} + \theta^2\theta_{\mu}]dr>0.$$
			This is impossible. Thus the claim is proved.
			
			Working on the interval $(0,\eta )$ similarly, one can prove that if $\theta_{\mu}'(\eta )\leq0$, then $\theta_{\mu}\equiv0$ in $[0,\eta ]$.
From these assertions
 and  $\theta_{\mu}$-equation \eqref{equation31} we have $\theta''\equiv0$ in $[0, \eta]$ or $[\eta, 1]$.
 As $\theta'(0)=\theta'(1)=0$,  $\theta'\equiv0$ in $[0, \eta]$ or $[\eta, 1]$,
which  contradicts the assumption. 
		\end{proof}
	
\medskip

\noindent{\textbf{Proof of Theorem \ref{deincreasing}.}}
Given any $\bar{\mu}>0$, by Lemmas \ref{once} and \ref{prepare} we can conclude that
	\begin{equation}\label{LiLou33}
	\theta_{\mu}(\bar{x};\bar{\mu})<0
\end{equation}
 holds for $\bar{x}$ satisfying $\theta(\bar{x},\bar{\mu})=M(\bar{\mu})$.
 The rest of the proof is similar to that of Theorem \ref{doublem}.
 \qed

\section{\bf Discussions}

In this paper, as motivated by the investigation of a predator-prey model in heterogeneous
environments \cite{LouYuan-WangBiao}, we studied whether the maximum of the unique
solution of \eqref{EQmain} is a monotone decreasing function of the diffusion rate.
For several classes of resource functions we proved that this is indeed the case.
In contrast,
the minimum of the unique
solution of \eqref{EQmain}
is not always  monotone increasing in the diffusion rate for
general resource functions
\cite{HeXiaoqing-NiWeiMing2016}. In fact, it is
quite curious that for large diffusion rate
it could occur that the density of the population is greater than
the average of the resource function everywhere in the whole habitat.

A probably interesting and related question is whether $\int_\Omega \theta^p(x; \mu)\,dx$
is strictly monotone decreasing in $\mu$.
That is, if $\int_\Omega \theta^p(x; \mu)\,dx$ were monotone decreasing in $\mu$
for all large $p$, then $\max_{\bar\Omega}\theta$ is also monotone decreasing in $\mu$
by applying the well-known
limit $\max_{\bar\Omega} \theta=\lim_{p\to \infty} \|\theta\|_{L^p(\Omega)}$.

For $p=1$,  it is shown in \cite{lou-2006}
that for the unique solution $\theta(x; \mu)$
of \eqref{EQmain}, the total biomass, given by the integral $\int_\Omega \theta(x; \mu)\, dx$,
is generally not a monotone function of the diffusion rate $\mu$.
In fact, it is possible to construct examples such that $\int_\Omega \theta(x; \mu)\, dx$ has multiple
critical points \cite{LandL}.
We refer to \cite{Bai2014, DeAngelis2015, NY} and references therein for more recent developments.

While it is unknown whether $\int_\Omega \theta^p(x; \mu)\,dx$ is monotone decreasing in $\mu$
for general $p$, the answer is affirmative for $p=3$ as shown in the following result:

\begin{lemma} $\int_\Omega \theta^3(x; \mu)\, dx$ is strictly monotone decreasing in $\mu$.
\end{lemma}
 \begin{proof} Differentiating \eqref{EQmain} with respect to $\mu$ we have
 \begin{equation}\label{EQmain-1}\begin{cases}
\mu \Delta \theta_\mu+ \theta_\mu(m(x)-2\theta)=-\Delta\theta \quad \mbox{in}\ \Omega,\\
\frac{\partial\theta_\mu}{\partial n}=0\quad  \mbox{on} \ \partial \Omega.
\end{cases}\end{equation}

 Multiplying \eqref{EQmain-1} by $\theta$, \eqref{EQmain} by $\theta_\mu$ and
 subtracting, and integrating the result in $\Omega$
 we have
 $$
 \int_\Omega \theta^2 \theta_\mu=\int_\Omega \theta \Delta\theta
 =-\int_\Omega |\nabla \theta|^2.
 $$
 Therefore
 $$
 \frac{d}{d\mu}\int_\Omega \theta^3
 =3\int_\Omega \theta^2 \theta_\mu
 =-3\int_\Omega |\nabla \theta|^2<0,
 $$
where the last inequality follows
from the fact that $\theta$ is non-constant.
 \end{proof}

It will be of interest to see whether for general resource function $m$,
$\int_\Omega \theta^p(x; \mu)\,dx$ is monotone decreasing in $\mu$
for $p\ge 3$. 

\bigskip

\noindent{\bf Acknowledgment.} The authors would like
to thank two anonymous referees for their helpful comments
which improve the presentation of the paper.
The authors are grateful to Prof. Bei Hu for his helpful comments and suggestions.
RL is sponsored by the China Scholarship Council
and she wishes to thank the Department of Applied Computational Mathematics and Statistics of the University of Notre Dame for the warm hospitality during her visit.
RL is partially supported by NSFC grants No. 11571364 and 11571363
and YL is partially supported by
 NSF grant DMS-1411476.

\end{document}